\documentclass[12pt]{article}
\usepackage[centertags]{amsmath}
\usepackage{amsfonts}
\usepackage{amssymb}
\usepackage{amsthm}
\input{amssym.def}
\usepackage{graphicx}

\numberwithin{equation}{section}

\newtheorem{Definition}{Definition}[section]
\newtheorem{theorem}[Definition]{Theorem}

\begin{document}
\title{\Large \bf On the power graphs which are Cayley graphs of some groups}
\author{S. Mukherjee and A. K. Bhuniya}
\date{}
\maketitle

\begin{center}
Department of Mathematics, Visva-Bharati, Santiniketan-731235, India. \\
shyamal.sajalmukherjee@gmail.com and anjankbhuniya@gmail.com
\end{center}

\begin{abstract}
In $2013$, Jemal Abawajy, Andrei Kelarev and Morshed Chowdhury \cite{Kelarev} proposed a problem to characterize  the finite groups whose power graphs are Cayley graphs of some groups. Here we give a complete answer to this question.
\end{abstract}

\section{Introduction}
Let $G$ be a finite group. The concept of directed power graph $\mathcal{\overrightarrow{P}}(G)$ was introduced by Kelarev and Quinn \cite{kelarev 2} . $\mathcal{\overrightarrow{P}}(G)$ is a digraph with the vertex set $G$ and for $x,y \in G $, there is an arc from $x$ to $y$ if and only if $x \neq y$ and $y = x^m$ for some positive integer $m$. Then Chakraborty, Ghosh and Sen \cite{sen} defined undirected power graph $\mathcal{P}(G)$ of a group $G$, which is a simple graph having the same vertex set as $\mathcal{\overrightarrow{P}}(G)$ and two distinct vertices in $G$ are adjacent if and only if either of them is a positive power of the other. The undirected power graph $\mathcal{P}(G)$ of a finite group $G$ is complete if and only if $G$ is a cyclic $p-$group [Theorem 2.12, \cite{sen}]

Cayley graph is another and widely studied graph, associated with finite groups. Let $G$ be a group and $C$ be a subset of $G\setminus \{e\}$. Then the directed Cayley graph $\mathcal{\overrightarrow{X(G,C)}}$ is defined to be a directed graph with vertex set $G$ and arc set $\{(g,h):g^{-1}h \in C \}$. If in addition, $C$ is an inverse closed subset, then the undirected Cayley graph $X(G,C)$ is defined to be the underlying undirected graph of $\mathcal{\overrightarrow{X(G,C)}}$. We refer to \cite{Godsil} for more on Cayley Graph of a group.

In the excellent survey \cite{kelarev 2} on power graph of groups and semigroups, Abawajy et. al. proposed the following problem:
\begin{quote}
\textit{Describe all directed and undirected power graphs of groups and semigroups that can be represented as Cayley graphs and, respectively, undirected Cayley graphs of groups or semigroups.}
\end{quote}

Here we describe all groups whose power graphs can be represented as Cayley graphs of some groups, both directed and undirected.

\section{Main result}
An undirected graph $\Gamma$ is called regular if degree of each vertex is the same and $\Gamma$ is called vertex transitive if its automorphism group acts transitively on the vertex set of $\Gamma$. Every Cayley graph is vertex transitive and every vertex transitive undirected graph $\Gamma$ is regular. If $\Gamma$ is directed vertex transitive graph then the in-degree of all the vertices are equal and the same holds for out-degree. Thus we have:
\begin{theorem}
\begin{enumerate}
\item
Let $G$ be a finite group.Then $\mathcal{P}(G)$ is a Cayley graph If and only if $G$ is a cyclic $p-$group.
\item
There is no finite group $G$, whose directed power graph is a Cayley graph.
\end{enumerate}
\end{theorem}

\begin{proof}
1. \; Suppose $G$ is a cyclic $p-$group, then the undirected power graph $\mathcal{P}(G)$ is complete and hence is a Cayley graph.

Conversely, assume that $\mathcal{P}(G)$ is a Cayley graph of some group. then $\mathcal{P}(G)$ is vertex transitive and so regular. If the order of $G$ is $n$ and $e$ is the identity element of $G$, then degree of $e$ is $n-1$ in $\mathcal{P}(G)$ and hence degree of $v$ is $n-1$ for every $v$ in $G$. It follows that $\mathcal{P}(G)$ is complete. Therefore $G$ is a cyclic $p-$group.
\\ 2. \; If possible, on the contrary, assume that $\mathcal{\overrightarrow{P}}(G)$ is a directed Cayley graph for some group $G$ of order $n$. Then $\mathcal{\overrightarrow{P}}(G)$ is vertex transitive and hence in-degree of all the vertices are equal and the same holds for out-degree. Now in $\mathcal{\overrightarrow{P}}(G)$, out-degree of $e$ is $0$ and in-degree of $e$ is $n-1$. Hence it follows that the out-degree of each vertex is $0$ and the in-degree of each vertex is $n-1$, which is impossible. Therefore, there is no finite group $G$ such that the directed power graph $\mathcal{\overrightarrow{P}}(G)$ is Cayley graph of some group.
\end{proof}

\bibliographystyle{amsplain}

\end{document}